\documentclass[12pt]{article}
\usepackage{mathrsfs}
\usepackage{amssymb}
\usepackage{}
\usepackage{amsfonts}
\usepackage{amsfonts,amsmath,amsthm, amssymb}
\usepackage{array}
\usepackage{latexsym, euscript, epic, eepic}
\usepackage{tikz}
\usepackage[noblocks]{authblk}
\usetikzlibrary{matrix}
\newtheorem{lemma}{Lemma}
\newtheorem{theorem}{Theorem}

\newtheorem{corollary}{Corollary}

\begin{document}
 
\title{BMO-Teichm\"uller spaces revisited}

\author{Huaying Wei \thanks{Department of Mathematics and Statistics, Jiangsu Normal University, Xuzhou 221116, PR China. Email:  6020140058@jsnu.edu.cn. Research supported by the National Natural Science Foundation of China (Grant No. 11501259) and the Natural Science Foundation of Colleges of Jiangsu Province (Grant No. 15KJB110006).} ,   Michel Zinsmeister \thanks{Corresponding author.\; MAPMO, Universit\'e d' Orl\'eans, Orl\'eans  Cedex 2, France. Email: zins@univ-orleans.fr}}
\date{}
\maketitle

\begin{center}
\begin{minipage}{120mm}
{\small{\bf Abstract}.
In \cite{CZ} the equivalence among three definitions of BMO-Teichm\"uller spaces associated with a Fuchsian group was proven using the Douady-Earle extension operator. In this paper, we show that these equivalences are actually biholomorphisms. In \cite{CZ} it was further shown that the Douady-Earle extension operator is continuous at the origin. We improve this result  by showing  G\^ateaux-differentiability at this point.
}

\end{minipage}
\end{center}


 \vskip1cm

\section{Introduction}
Let $h$ be a quasisymmetric homeomorphism of the unit circle $\mathbb{S}$. Ahlfors and Beurling \cite{BA} have been the first to prove that $h$ may be extended to a quasiconformal homeomorphism of the unit disk $\mathbb{D}$. Later Douady and Earle \cite{DE} found a conformally natural way to extend $h$ to  a quasiconformal homeomorphism of $\mathbb{D}$. More precisely, their extension, called Douady-Earle extension (or barycentric extension), denoted by $E(h)$, satisfies 
\begin{equation*}
E(\tau \circ h \circ \alpha) = \tau \circ E(h) \circ \alpha
\end{equation*}
for any couple $\tau$, $\alpha$ of automorphisms of $\mathbb{D}$. The Douady-Earle extension plays an important role applied to quasisymmetric homeomorphisms of $\mathbb{S}$ in the complex analytic theory of Teichm\"uller spaces. In this paper, our study on BMO-Teichm\"uller theory is based on good properties of the Douady-Earle extension.

The universal Teichm\"uller space $T$ can be defined as the space $QS_{*}(\mathbb{S})$ of all normalized quasisymmetric homeomorphisms of $\mathbb{S}$. In this setting, the Teichm\"uller projection $\varPhi$ is regarded as the boundary extension map on the space $QC_{*}(\mathbb{D})$ of all normalized quasiconformal homeomorphisms of $\mathbb{D}$. By the measurable Riemann mapping theorem, we can identity the latter space with the space of Beltrami coefficients $M(\mathbb{D}) = L^{\infty}(\mathbb{D})_1$, which is the open unit ball of measurable functions on $\mathbb{D}$ with the supremum norm. Then $\varPhi: M(\mathbb{D}) \to T$ is continuous with respect to the topology on $QS_{*}(\mathbb{S})$ induced by the quasisymmetry constant. The Douady-Earle extension yields a continuous section $e: T \to M(\mathbb{D})$ for $\varPhi$. This section is called the Douady-Earle extension operator which maps the quasisymmetric homeomorphism $h$ to the complex dilatation $\mu$ of the Douady-Earle extension $E(h)$ of $h$. The continuity of this section $e$ combined with conformally natural property of the Douady-Earle extension allowed Douady and Earle \cite{DE} to give a much simpler proof of the theorem of Tukia \cite{Tu1, Tu2} stating that the Teichm\"uller space of any Fuchsian group is contractible.

The universal BMO-Teichm\"uller space $T_b$ is similarly defined as a subspace of $T$.  It is defined as the  subspace $SQS_{*}(\mathbb{S}) \subset  QS_{*}(\mathbb{S})$  of all normalized strongly quasisymmetric homeomorphisms. The topology on $SQS_{*}(\mathbb{S})$ is induced by the BMO norm. On the other hand, the corresponding subspace of Beltrami coefficients is $\mathcal{M}(\mathbb{D}) \subset M(\mathbb{D})$, which consists of all $\mu \in M(\mathbb{D})$ such that 
\begin{equation*}
\frac{|\mu|^2(z)}{1 - |z|^2} dxdy 
\end{equation*}
is a Carleson measure in $\mathbb{D}$. Let's consider a Fuchsian group $G$: define $\mathcal{M}(G) = M(G) \cap \mathcal{M}(\mathbb{D})$, $SQS_*(G) = QS(G) \cap SQS_*(\mathbb{S})$. The same equivalence relation as in the classical case may be defined on $\mathcal{M}(G)$ and we denote by $\mathcal{T}_S$ the quotient space 
($S$ is the Riemann surface $\mathbb{D}/G$). Let $T(G)$ be the space of Schwarzian derivatives of injective holomorphic functions in $\hat{\mathbb{C}}\setminus\overline{\mathbb{D}}$ having a quasiconformal extension to $\mathbb{C}$. Define $\mathcal{T}(G) = \{\varphi \in T(G);\, |\varphi|^2(z) (|z|^2 - 1)^3 dxdy$ is a Carleson measure on $
\hat{\mathbb{C}}\setminus\overline{\mathbb{D}} \}$.

Cui and Zinsmeister have proved in \cite{CZ} that for any $h \in SQS_{*}(\mathbb{S})$ the complex dilatation $\mu$ of its Douady-Earle extension is in 
$\mathcal{M}(\mathbb{D})$.  Based on this well-defined property of the Douady-Earle extension operator restricted in $SQS_{*}(\mathbb{S})$, Cui and Zinsmeister have shown that the map $\varPsi: \mathcal{T}_S \to SQS_*(G)$ is a bijection while the Bers embedding $\beta: \mathcal{T}_S \to \mathcal{T}(G)$ is bijective. In section 3, we will show that complex Banach manifold structures can be provided for $\mathcal{T}_S$ and $SQS_*(G)$ through the Bers embedding $\beta$ and the map $\beta \circ \varPsi^{-1}$.  Then both the map $\varPsi: \mathcal{T}_S \to SQS_*(G)$ and the Bers embedding $\beta: \mathcal{T}_S \to \mathcal{T}(G)$ become  biholomorphic.

By Cui and Zinsmeister $e(h) \in \mathcal{M}(\mathbb{D})$ if $h \in SQS_{*}(\mathbb{S})$, and moreover $e$ is continuous at the origin. The global continuity is  not known. This property would imply contractibility of all BMO-Teichm\"uller spaces; So far only the case $G = \{I\}$ is known \cite{FHS}. In section 4, it is proven that at least the operator $e$ is also G\^ateaux-differentiable at the origin, and we identify its differential which happens to be a very simple operator.

In section 2, we will explain the above mentioned concepts and results in more detail.

\section{Preliminaries}
In this section, we summarize several results on the background of our arguments. This includes definitions and properties of Teichm\"uller spaces, preliminaries on BMO-Teichm\"uller spaces and fundamental results on groups of divergence type and groups of convergence type.

\vspace{4mm}

\noindent\textbf{2.1 Teichm\"uller theory.}  \,Let G be a Fuchsian group, i.e. a properly discontinuous fixed point free group of M\"obius transformations which keeps  
$\mathbb{D}$ invariant. For such a group we define $M(G)$ as
$$M(G) = \{\mu \in L^{\infty}(\mathbb{D}): \|\mu\|_{\infty} < 1\,\, and\,\, \forall g \in G,\, \mu = \mu \circ g \frac{\overline{g'}}{g'}\}.$$
For any $\mu \in M(G)$, there exists a unique quasiconformal self-mapping $f^{\mu}$ of $\mathbb{D}$ keeping 1, $i$ and -1 fixed and satisfying
$$\frac{\partial f^{\mu}}{\partial\bar{z}} = \mu \frac{\partial f^{\mu}}{\partial z}$$
in $\mathbb{D}$.
Similarly, there exists a unique quasiconformal homeomorphism of $\hat{\mathbb{C}}$ which is holomorphic in  $\mathbb{D}^{*}$ with the normalization
$$f_{\mu}(z) = z + \frac{b_{1}}{z} + \cdots$$
at $\infty$ and such that  
$$\frac{\partial f_{\mu}}{\partial\bar{z}} = \mu \frac{\partial f_{\mu}}{\partial z}$$
in $\mathbb{D}$. 
If $g$ is a choice of a Riemann mapping from $\mathbb{D}$ onto $\Omega$, $\Omega = f_\mu (\mathbb{D})$, then $f^{\mu} = g^{-1}\circ f_{\mu}$ is the conformal welding with respect to the boundary of the domain $\Omega$. The mappings $f^{\mu}$ and $f_{\mu}$ respectively induce an isomorphism of the group G onto the Fuchsian group
$$G^{\mu} = \{f^{\mu}\circ g \circ(f^{\mu})^{-1} \mid g \in G\}$$
and the quasi-Fuchsian group
$$G_{\mu} = \{f_{\mu}\circ g \circ(f_{\mu})^{-1} \mid g \in G\},$$
i.e., a M\"obius transformation group acting properly discontinuous on the quasidisk $f_{\mu}(\mathbb{D})$.

The mapping $f^{\mu}$ has a geometric interpretation: If we denote by $S$ the Riemann surface $\mathbb{D}/G$, then $f^{\mu}$ is the lift (to the universal covering) of a quasiconformal mapping from the Riemann surface $S$ onto $S^{'} = \mathbb{D}/G^{\mu}$. Conversely, if F is a quasiconformal homeomorphism from S to a Riemann surface $S^{'}$, it has a lift to a quasiconformal homeomorphism $f$ of $\mathbb{D}$ and, replacing if necessary $F$ by $\theta\circ F$, where $\theta: S^{'} \rightarrow S^{''}$ is a conformal isomorphism, we may assume that $f = f^{\mu}$ for some $\mu \in M(G)$.

If $\mu \in M(G)$, then $f^{\mu}$ has a well-defined boundary value which is a quasisymmetric homeomorphism of $\mathbb{S}$. We define an equivalence relation on $M(G)$ by $\mu\sim\nu$ if $f^{\mu}|_{\mathbb{S}} = f^{\nu}|_{\mathbb{S}}$. Again this equivalence relation has a geometric interpretation: If F, G represent the quasiconformal mappings on S whose lifts are precisely $f^{\mu}$, $f^{\nu}$, then $\mu\sim\nu$ is equivalent to saying that $G\circ F^{-1}$ is homotopic to a conformal isomorphism between $F(S)$ and $G(S)$, the homotopy being constant on the (possibly empty) boundary of $F(S)$.

The Teichm\"uller space $T_S$ is the quotient space $M(G)/\sim$ . We refer to \cite{Le} for details about this construction.

If $\mu \in M(G)$, then the Teichm\"uller space $T_S$ can be characterized as the set of quasisymmetric homeomorphisms $f^{\mu}|_{\mathbb{S}}$. Since $\mu \in M(G)$, the mappings $f^{\mu}$ and $f^{\mu}\circ g$ have the same complex dilatation. It follows that $f^{\mu}\circ g\circ (f^{\mu})^{-1}$ is a M\"obius transformation. It is well known that $\Psi: [\mu]\mapsto f^{\mu}|_{\mathbb{S}}$ is a bijection from $T_S$ onto $QS_*(G)$, the set of quasisymmetric homeomorphisms $h$ of 
$\mathbb{S}$ keeping 1, $i$ and $-1$ fixed and  such that $h\circ g \circ h^{-1}$ is a M\"obius transformation.

There is a similar description of the Teichm\"uller space in terms of $f_{\mu}$.
Let the Banach space B(G) be the space consisting of all functions $\varphi$ holomorphic in $\mathbb{D}^{*}$ which are quadratic differentials for G and have a finite hyperbolic supremum norm:
$$\|\varphi\|_{B} = \sup_{z \in \mathbb{D}^*}|\varphi(z)|\rho_{\mathbb{D}^{*}}^{-2}(z) < \infty.$$
If $\mu \in M(G)$, then the quasiconformal mapping $f_{\mu}\circ g \circ f_{\mu}^{-1}$, $g \in G$ of the plane is a M\"obius transformation. It follows that
$$\mathcal {S}_{f_{\mu}|_{\mathbb{D}^{*}}} = \mathcal {S}_{(f_{\mu}\circ g \circ f_{\mu}^{-1})\circ f_{\mu}|_{\Delta^{*}}} = \mathcal {S}_{f_{\mu}\circ g|_{\mathbb{D}^{*}}} = (\mathcal {S}_{f_{\mu}|_{\Delta^{*}}}\circ g)(g^{'})^2.$$
So the Schwarzian derivative $\mathcal {S}_{f_{\mu}|_{\mathbb{D}^{*}}}$ is a quadratic differential for G. It is also well known that the Bers embedding $\beta:  [\mu]\mapsto \mathcal {S}_{f_{\mu}|_{\mathbb{D}^{*}}}$ is a bijection from $T_S$ onto $T(G)$, the space of Schwarzian derivatives of injective holomorphic functions in $\mathbb{D}^{*}$ having a quasiconformal extension to the complex plane which are quadratic differentials for G. It is known that the set $T(G)$ is an open subset in the complex Banach space $B(G)$, and the ball
$$B(0, 2) = \{\phi \in B(G): \|\phi\|_{B} < 2\}$$
lies in $T(G)$. With the aid of the Bers embedding, $T_S$ carries a natural complex structure. The Teichm\"uller space $T_S$ thus becomes a complex analytic Banach manifold. We refer to \cite{Le} for details about the Teichm\"uller theory.

\vspace{4mm}

\noindent\textbf{2.2 BMO-Teichm\"uller theory.}  \,Recall that a positive measure $\lambda$ defined in a simply connected domain $\Omega$ is called a Carleson measure (see \cite{Ga}) if
\begin{equation}\label{carlesonnorm}
\|\lambda\|_{c} = \sup \{\frac{\lambda(\Omega \cap D(z, r))}{r}: z \in \partial \Omega, 0 < r < diameter(\Omega)\} < \infty,
\end{equation}
where $D(z, r)$ is the disk with center $z$ and radius $r$. A Carleson measure $\lambda$ is called a vanishing Carleson measure if $\lim_{r\rightarrow 0}\lambda(\Omega \cap D(z, r))/r = 0$ uniformly for $z \in \partial \Omega$. We denote by $CM(\Omega)$ and $CM_0(\Omega)$ the set of all Carleson measures and vanishing Carleson measures on $\Omega$, respectively.

We denote by $\mathcal {L}(\mathbb{D})$ the Banach space of  essentially bounded measurable functions $\mu$ on $\mathbb{D}$ such that the measure 
$$\lambda_{\mu} = \frac{|\mu|^2(z)}{1 - |z|^2} dxdy$$ 
in $CM(\mathbb{D})$. The norm on $\mathcal {L}(\Delta)$ is defined as
$$\|\mu\|_c = \|\mu\|_{\infty} + \|\lambda_{\mu}\|_c^{1/2},$$
where $\|\lambda_{\mu}\|_c$ is the Carleson norm of $\lambda_{\mu}$ defined in (\ref{carlesonnorm}). Set $\mathcal {M}(\mathbb{D}) = \{\mu \in \mathcal {L}(\mathbb{D}): \|\mu\|_{\infty} < 1\}$. Define $\mathcal {M}(G) = M(G)\cap \mathcal {M}(\mathbb{D})$. The same equivalence relation as in the classical case may be defined on $\mathcal {M}(G)$ and we denote by $\mathcal {T}_S$ the quotient space which can be called BMO-Teichm\"uller spaces.

An homeomorphism $h$ of $\mathbb{S}$ is called strongly quasisymmetric (see \cite{Jo}) if there exist two positive constants $C_1(h)$, $C_2(h)$, called the strongly quasisymmetric constants of $h$ such that
$$\frac{|h(E)|}{|h(I)|}\leqslant C_1(h) (\frac{|E|}{|I|})^{C_2(h)}$$
whenever $I \subset \mathbb{S}$ is an interval and $E \subset I$ a measurable subset. In other words, $h$ is strongly quasisymmetric if and only if $h$ is absolutely continuous so that $|h^{'}|$ belongs to the class of weights $A^{\infty}$ introduced by Muckenhoupt, in particular, $\log h^{'}$ belongs to $BMO(\mathbb{S})$, the space of integrable functions on $\mathbb{S}$ of bounded mean oscillation. Let $SQS(\mathbb{S})$ denote the set of all  strongly quasisymmetric homeomorphisms on $\mathbb{S}$. We define $SQS_*(G) = QS_*(G) \cap SQS(\mathbb{S})$.

We denote by $\mathcal {B}(\mathbb{D}^*)$ the Banach space of function $\varphi$ holomorphic in $\mathbb{D}^*$ such that the measure 
$$\lambda_{\varphi} = |\varphi(z)|^2 (|z|^2 - 1)^3 dxdy$$
 in $CM(\mathbb{D}^*)$. The norm on $\mathcal {B}(\mathbb{D}^*)$ is
$$\|\varphi\|_{\mathcal {B}} = \|\lambda_{\varphi}\|_c.$$
Define $\mathcal {T}(G) = \{\varphi \in T(G): \lambda_{\varphi} \in CM(\mathbb{D}^*)\}$.
Let $\mathcal {B}(G)$ be the space consisting of all functions $\varphi$ in $\mathcal {B}(\mathbb{D}^*)$ which are quadratic differentials for G. We can see that $\mathcal {T}(G)$ is a subset of $\mathcal {B}(G)$.

In 2004, Cui and Zinsmeister \cite{CZ} proved the following theorem based on the well-defined property of the Douady-Earle extension operator restricted in $SQS_{*}(\mathbb{S})$.

\vspace{2mm}

\noindent \textbf{Theorem A.}\,The mapping $\Psi: [\mu] \mapsto f^{\mu}$ is a bijection from $\mathcal {T}_S$ onto $SQS_*(G)$, while $\beta: [\mu]\mapsto \mathcal {S}_{f_{\mu}}$ is bijective from $\mathcal {T}_S$ onto $\mathcal {T}(G)$.

\vspace{2mm}

One of the goals of this paper is to make this theorem precise using methods in \cite{SW} by showing both  two maps above are actually homeomorphisms (even  biholomorphic automorphisms).  Before proceeding we end this section by a discussion on Fuchsian groups leading to non-trivial BMO-Teichm\"uller  spaces

\vspace{4mm}

\noindent\textbf{2.3  Groups of convergence type.}  \,In contrast to the classical Teichm\"uller spaces, $\mathcal {T}_S$ can be trivial, as shown by Astala and Zinsmeister \cite{AZ2}. For completeness, let us recall some related facts. 

Let $G$ be a discrete group of M\"obius transformation on $ \hat{\mathbb{C}}$. We say that $G$ has the Mostow rigidity property if for each homeomorphism $h: \mathbb{S} \rightarrow \mathbb{S}$ with $h \circ G \circ h^{-1}$ a M\"obius group, it holds that either $h$ is completely singular or else is a M\"obius transformation.

We say that the group $G$ acting on $ \hat{\mathbb{C}}$ is of divergence type if 
$$\sum_{\gamma \in G}(1 - |\gamma (0)|) = \infty,$$
and of convergence type, if the series converge. It was shown by Agard \cite{Ag} and Sullivan \cite{Su} that groups of divergence type have the Mostow rigidity property.  In 1990, Astala and Zinsmeister \cite{AZ2} proved the converse:

\vspace{2mm}

\noindent\textbf{Theorem B. }\,A Fuchsian group $G$ has Mostow rigidity property if and only if it is of divergence type.

\vspace{2mm}

\noindent Combining with Theorem A, we see that $\mathcal {T}_S$ is trivial if $G$ is of divergence type, while $\mathcal {T}_S$ is never trivial if $G$ is of convergence type. In what follows, we henceforth shall not deal with Fuchsian groups of divergence type. We shall focus on Fuchsian groups of convergence type.

\section{Complex structure on $\mathcal {T}_S$}

In this section,  we adopt methods from \cite{SW} to prove that $\mathcal {T}_S$ has a natural complex Banach manifold structure. As a byproduct, we shall strengthen the above conclusion (Theorem A) by Cui and Zinsmeister. 

We begin with some basic definitions and notations. Let $C$  denote the universal constant that might change from one line to another. While $C(\cdot)$, $C_1(\cdot)$, $C_2(\cdot)$, $\cdots$ will denote constants that depend only on the elements put in the brackets. Denote by $\rho_{D}(z)$ the hyperbolic metric in the Jordan domain $D$. The notation $A \approx B$ means that there exists a universal constant $C$ such that $\frac{B}{C} \leqslant A \leqslant CB$.

Let $\Omega = f_{\mu}(\mathbb{D})$ and $\Omega^{*} = f_{\mu}(\mathbb{D}^*)$, then $\Omega$ and $\Omega^*$ are complementary Jordan regions bounded by a quasicircle. Let
\begin{equation}\label{reflection}
z^{*} = \gamma(\Omega)(z) = f_{\mu}\circ j \circ f_{\mu}^{-1}(z)
\end{equation}
be a quasiconformal reflection that fixes $\partial \Omega$ pointwise and interchanges $\Omega$ and $\Omega^*$, where $j(z) = \frac{1}{\bar{z}}$. If $g$ is a M\"obius transformation of the extended complex plane, then,  according to (\ref{reflection}),
$$\gamma(g(\Omega)) = (g\circ f_{\mu})\circ j \circ (g \circ f_{\mu})^{-1} = g\circ \gamma(\Omega)\circ g^{-1}.$$
Similar to the Ahlfors map (see \cite{EN}), the mapping 
$$\varphi \mapsto \mu_{\Omega}(\varphi)$$ 
is defined on the Banach space 
$$B(\Omega^{*}) = \{\varphi \, \text{ is holomorphic in } \Omega^{*}:  \sup_{z \in \Omega^*}|\varphi(z)|\rho_{\Omega^{*}}^{-2}(z) < \infty\}$$
by putting
\begin{equation}\label{Ahlfors mapping}
\mu_{\Omega}(\varphi)(z) = \frac{\varphi(z^*)(z^* - z)^2 \gamma(\Omega)_{\bar{z}}(z)}{2 + \varphi(z^*)(z^* - z)^2 \gamma(\Omega)_{z}(z)},\,\, z \in \Omega.
\end{equation}

\begin{lemma}\label{composition}
$\mu_{g(\Omega)}(\varphi)(g(z))\frac{\overline{g^{'}(z)}}{g^{'}(z)} = \mu_{\Omega}(\varphi\circ g (g^{'})^2)(z)$.
\end{lemma}

\begin{proof}
We borrow the method from Earle-Nag \cite{EN}. According to (\ref{Ahlfors mapping}),
$$\mu_{g(\Omega)}(\varphi)(g(z))\frac{\overline{g^{'}(z)}}{g^{'}(z)} = \frac{\varphi(g(z)^*)(g(z)^* - g(z))^2\gamma(g(\Omega))_{\bar{z}}(g(z))}{2 + \varphi(g(z)^*)(g(z)^* - g(z))^2\gamma(g(\Omega))_{z}(g(z))}\frac{\overline{g^{'}(z)}}{g^{'}(z)}.$$
Here $g(z)^* = \gamma(g(\Omega))(g(z)) = g\circ \gamma(\Omega)\circ g^{-1}(g(z)) = g\circ \gamma(\Omega)(z) = g(z^*)$, if $z^* = \gamma(\Omega)(z)$.
Differentiating both sides of
$$\gamma(g(\Omega))(g(z)) = g\circ\gamma(\Omega)(z)$$
with respect to $\bar{z}$ and $z$ respectively, we find that
$$\gamma(g(\Omega))_{\bar{z}}(g(z))\overline{g^{'}(z)} = g^{'}(z^{*})\gamma(\Omega)_{\bar{z}}(z)$$
and
$$\gamma(g(\Omega))_{z}(g(z))g^{'}(z) = g^{'}(z^{*})\gamma(\Omega)_{z}(z).$$
Therefore,
\begin{equation*}
\begin{split}
\mu_{g(\Omega)}(\varphi)(g(z))\frac{\overline{g^{'}(z)}}{g^{'}(z)} & = \frac{\varphi(g(z^{*}))(z^{*} - z)^2 g^{'}(z^{*})^2\gamma(\Omega)_{\bar{z}}(z)}{2 + \varphi(g(z^{*}))(z^{*} - z)^2 g^{'}(z^{*})^2\gamma(\Omega)_{z}(z)}\\
 & = \mu_{\Omega}(\varphi\circ g (g^{'})^2)(z).\\
\end{split}
\end{equation*}
\end{proof}

\begin{lemma}[see \cite{SW}]\label{CM}
Let $\alpha > 0$,  $\beta > 0$. For a positive
function $\lambda$ in $\mathbb{D}$, set
\begin{equation}
\tilde{\lambda}(z) = \iint_{\mathbb{D}}
\frac{(1-|z|^2)^{\alpha}(1-|w|^2)^{\beta}}{|1-\bar{z}w|^{\alpha +
\beta + 2}}\lambda (w) dudv.
\end{equation}
Then, using the same notation for a function $\tau$ and the associated measure $\tau dxdy$, we have that  $\tilde{\lambda} \in \text{CM}(\mathbb{D})$ if $\lambda \in
\text{CM}(\mathbb{D})$, and $\| \tilde{\lambda} \|_c
\leqslant C
\| \lambda \|_c $, while $\tilde{\lambda} \in
\text{CM}_0(\mathbb{D})$ if $\lambda \in \text{CM}_0(\mathbb{D})$.
\end{lemma}

\begin{lemma}[see \cite{Zi89}]\label{Zinsmeister1989}
Let $f$ be conformal in  $\mathbb{D}$. Then $\log f' \in \text{BMOA}(\mathbb{D})$ if and only if for each Carleson measure $\lambda \in \text{CM}(\Omega), \, \Omega = f(\mathbb{D})$, the inverse image $\lambda\circ f|f'| \in \text{CM}(\mathbb{D})$.
Besides, the norm of $\lambda\circ f|f'|$ is dominated by the norm of $\lambda$.
\end{lemma}

Let us denote by $T(1)$, $\mathcal{T}(1)$, the spaces $T(G)$, $\mathcal{T}(G)$ for $G = \{I\}$,
and also denote by $B(\mathbb{D}^{*})$, $\mathcal{B}(\mathbb{D}^{*})$, the spaces $B(G)$, $\mathcal{B}(G)$ for $G = \{I\}$. We will also need the spaces  $B(\mathbb{D})$ and $\mathcal{B}(\mathbb{D})$.
 $B(\mathbb{D})$ and $\mathcal{B}(\mathbb{D})$ can be defined respectively just $\rho_{\mathbb{D}^{*}}(z)$ replaced by $\rho_{\mathbb{D}}(z)$ in the definitions of $B(\mathbb{D}^{*})$ and $\mathcal{B}(\mathbb{D}^{*})$.
 The following Lemma shows that $\mathcal{B}(\mathbb{D}) \subset B(\mathbb{D})$.
\begin{lemma}\label{unitdisk}
For any $\phi \in \mathcal{B}(\mathbb{D})$, $\|\phi\|_B \leqslant \frac{32\sqrt{2}}{3}\|\phi\|_{\mathcal{B}}$.
\end{lemma}
\begin{proof}
Examine the proof of  Lemma 4.1 in \cite{SW} carefully, we can find that this conclusion is valid.
\end{proof}

\begin{lemma}\label{Banach}
The space $\mathcal{B}(G)$ is a Banach space, and the set $\mathcal{T}(G)$ is open in $\mathcal{B}(G)$. Furthermore, there exists a constant $k$, so that the ball $B(0, k) = \{\varphi \in \mathcal{B}(G): \, \|\varphi\|_{\mathcal{B}} < k\}$ lies in $\mathcal{T}(G)$.
\end{lemma}

\begin{proof}
We consider functions $\varphi_n \in \mathcal{B}(G)$ which converge to $\varphi$ in $\mathcal{B}(\mathbb{D}^{*})$. Given a $g \in G$, we then have
$\varphi_n(z)\rightarrow \varphi(z)$, $\varphi_n(g(z))\rightarrow \varphi(g(z))$, uniformly on every compact subset of $\mathbb{D}^{*}$. It follows that
$$\varphi(g(z))g^{'}(z)^2 = \lim_{n \rightarrow \infty}\varphi_n(g(z))g^{'}(z)^2 = \lim_{n \rightarrow \infty}\varphi_n(z) = \varphi(z).$$
Consequently, $\varphi \in \mathcal{B}(G)$. Therefore, $\mathcal{B}(G)$ is a closed subspace of $\mathcal{B}(\mathbb{D}^{*})$ and hence a Banach space.

The definition of $\mathcal{T}(G)$  and Lemma \ref{unitdisk}  imply that
$$\mathcal{T}(1) = T(1)\cap \mathcal{B}(\mathbb{D}^{*}) \subset B(\mathbb{D}^{*})\cap \mathcal{B}(\mathbb{D}^{*}) =  \mathcal{B}(\mathbb{D}^{*}),$$
so $\mathcal{T}(1)$ is an open subset of $\mathcal{B}(\mathbb{D}^{*})$ by the openness of $T(1)$ in $B(\mathbb{D}^{*})$.
The definition of $\mathcal{T}(G)$ also implies that
$$\mathcal{T}(G) = \mathcal{T}(1)\cap \mathcal{B}(G) \subset \mathcal{B}(\mathbb{D}^{*})\cap \mathcal{B}(G) =  \mathcal{B}(G),$$
so $\mathcal{T}(G)$ is an open subset of Banach space $\mathcal{B}(G)$ by the openness of $\mathcal{T}(1)$ in $\mathcal{B}(\mathbb{D}^{*})$.

Now for $\varphi \in \mathcal{B}({\mathbb{D}^{*}})$, define $\phi = \varphi\circ g (g^{'})^2$, where $z = g(w) = \frac{1}{w}$, $w \in \mathbb{D}$. Then
$\|\phi\|_{B} = \|\varphi\|_{B}$. Let $\lambda_{\phi}(w) = |\phi(w)|^2(1 - |w|^2)^3$ and $\lambda_{\varphi}(z) = |\varphi (z)|^2(|z|^2 - 1)^3$  as above. then
$\lambda_{\phi}(\frac{1}{z})|\frac{1}{z}|^2 = \lambda_{\varphi}(z)$, $z \in \mathbb{D}^{*}$. So $\|\phi\|_{\mathcal{B}} \approx \|\varphi\|_{\mathcal{B}}$ by Lemma \ref{Zinsmeister1989}. According to Lemma \ref{unitdisk}, there is a constant $k$, such that $\|\varphi\|_B \leqslant \frac{2}{k}\|\varphi\|_{\mathcal{B}}$.
Suppose $\|\varphi\|_B \leqslant \frac{2}{k}\|\varphi\|_{\mathcal{B}} < 2$. Set $\mathcal{S}_{f} = \varphi$. According to Ahlfors-Weill Theorem, $f$ is univalent and can be extended to a quasiconformal mapping of the complex plane with complex dilatation
$$\mu(z) = -\frac{1}{z}\varphi(\frac{1}{\bar{z}})(1 - |z|^2)^2\frac{1}{(\bar{z})^4},\,\,z \in \mathbb{D}.$$
Lemma \ref{Zinsmeister1989} implies that
\begin{equation*}
\begin{split}
\lambda_{\varphi}(\frac{1}{z})|\frac{1}{z}|^2
& = |\varphi(\frac{1}{z})|^2(\frac{1}{|z|^2} - 1)^3\frac{1}{|z|^2}\\
& = |\varphi(\frac{1}{z})|^2(1 - |z|^2)^3\frac{1}{|z|^8}\\
\end{split}
\end{equation*}
is a Carleson measure in $\mathbb{D}$. Hence 
$$\frac{|\mu(z)|^2}{1 - |z|^2} = \frac{1}{4}|\varphi(\frac{1}{\bar{z}})|^2(1 - |z|^2)^3\frac{1}{|z|^8}$$
 is a Carleson measure in $\mathbb{D}$.
Consequently, the set $\{\varphi \in \mathcal{B}(\mathbb{D}^{*}): \, \|\varphi\|_{\mathcal{B}} < k\}$ lies in $\mathcal{T}(1)$.

For any $\varphi \in B(0, k)$, since $\mu(z)$  above can be written as
$$\mu(z) = \frac{1}{z}\varphi(z^{*})(z - z^{*})^2\gamma(\mathbb{D})_{\bar{z}}(z),$$
 where $z^{*} = \gamma(\mathbb{D})(z) = \frac{1}{\bar{z}}$. Then $\varphi\circ g (g^{'})^2 = \varphi$ implies $\mu \circ g \frac{\bar{g^{'}}}{g^{'}} = \mu$ by Lemma \ref{composition}.
So $\mu$ is a Beltrami differential for $G$. Now we can come to the conclusion that for any $\varphi \in B(0, k)$, there exists a holomorphic function $f$ in $\mathbb{D}^{*}$ which can be extended to a quasiconformal mapping in the complex plane whose complex dilatation $\mu \in \mathcal{M}(G)$ and $\mathcal{S}_{f} = \varphi$.
\end{proof}

\begin{theorem}\label{section}
The function
$$S: \mu \mapsto \mathcal{S}_{f_{\mu}|_{\Delta^{*}}}$$
which maps $\mathcal{M}(G)$ into $\mathcal{B}(G)$ is holomorphic, and it has local holomorphic sections everywhere in $\mathcal{T}(G)$.
\end{theorem}

\begin{proof}
We first show that $S: \mathcal{M}(G)\rightarrow \mathcal{B}(G)$ is continuous. We borrow some discussion from Astala and Zinsmeister \cite{AZ}. By an integral representation of the Schwarzian derivative by means of the representation theorem of quasiconformal mappings, Astala and Zinsmeister \cite{AZ} proved that for any two elements $\mu$ and $\nu$ in $M(\mathbb{D})$, there exists some constant $C_1(\|\mu\|_{\infty})$ such that 
$$|S(\nu)(z) - S(\mu)(z)|^{2} \leq \frac{C_1(\|\mu\|_{\infty})}{(|z|^2 - 1)^2} \iint_{\mathbb{D}}\frac{|\nu(\zeta) - \mu(\zeta)|^2 + \|\nu - \mu\|_{\infty}^2|\mu(\zeta)|^2}{|\zeta - z|^4}d\xi d\eta. $$
Then, 
\begin{equation*}
\begin{split}
 |S(\nu)(z) - S(\mu)(z)|^{2}&(|z|^2 - 1)^3 \\
& \leq C_1(\|\mu\|_{\infty})\Big( \iint_{\mathbb{D}}\frac{|\nu(\zeta) - \mu(\zeta)|^2}{1 - |\zeta|^2} \frac{(1 - |\zeta|^2)( |z|^2 - 1)}{|\zeta - z|^4} d\xi d\eta \\
& + \|\nu - \mu\|_{\infty}^2 \iint_{\mathbb{D}}\frac{|\mu(\zeta)|^2}{1 - |\zeta|^2}\frac{(1 - |\zeta|^2)(|z|^2 - 1)}{|\zeta - z|^4} d\xi d\eta \Big).\\
\end{split}
\end{equation*}
According to Lemma \ref{CM}, there exists some constant $C_2(\|\mu\|_{\infty})$ such that
\begin{equation*}
\begin{split}
\|S(\nu) - S(\mu)\|_{\mathcal {B}}^2  & \leq  C_2(\|\mu\|_{\infty})(\|\lambda_{\nu - \mu}\|^2_c + \|\nu - \mu\|_{\infty}^2\|\lambda_{\mu}\|^2_c)\\
& \leq  C_2(\|\mu\|_{\infty})(1 + \|\mu\|_c^2)\|\nu - \mu\|_c^2.\\
\end{split}
\end{equation*}
Consequently, $S: \mathcal {M}(G) \rightarrow \mathcal {B}(G)$ is continuous.

To prove that $S: \mathcal{M}(G)\rightarrow \mathcal{B}(G)$ is a holomorphic map, we use a general result about the infinite dimensional holomorphy (see \cite{Le, Na88}). It says that a continuous map $f$ from a domain $U$ in a complex Banach space $X$ into another complex Banach space $Y$ is holomorphic if for each pair $(u, x)$ in $U \times X$ and each element $y^{*}$ from a total subset $Y_{0}^{*}$ of the dual space $Y^{*}$, $y^{*}(f(u + tx))$ is a holomorphic function in a neighborhood of zero in the complex plane. Here a subset $Y_{0}^{*}$ of  $Y^{*}$ is total if $y^{*}(y) = 0$ for all $y^{*} \in Y_{0}^{*}$ implies that $y = 0$.

We have already seen from Lemma \ref{Banach} that $\mathcal{B}(G)$ is a Banach space. The ball $\mathcal{M}(G)$ is an open subset of the Banach space $\mathcal{L}(G)$. Now for each $z \in \mathbb{D}^{*}$, define $l_z(\phi) = \phi(z),\; \forall \phi \in \mathcal {B}(G).$ The proof of Lemma \ref{Banach} says that 
$|\phi(z)|(|z|^2 - 1)^2 \leq C\|\phi\|_{\mathcal {B}},$ 
which implies that 
$\|l_z\| \leq C(|z|^2 - 1)^{-2}.$ Thus, 
$l_z \in \mathcal {B}^{*}(G)$.
Set $A = \{l_z: z \in \mathbb{D}^{*}\}.$ Clearly, A is a total subset of 
$\mathcal {B}^{*}(G)$. Now for each $ z \in \mathbb{D}^{*}$, each pair $(\mu, \nu) \in \mathcal {M}(G) \times \mathcal {L}(G)$ and small $t$ in the complex plane, by the well known holomorphic dependence of quasiconformal mappings on parameters (see \cite{Ah66, Le, Na88}), we conclude that
$l_z(S(\mu + t\nu)) = S(\mu + t\nu)(z)$ is a holomorphic function of $t$. Consequently, 
$S: \mathcal {M}(G) \rightarrow \mathcal {B}(G)$ is holomorphic.

Finally, we prove $S: \mathcal{M}(G) \rightarrow \mathcal{B}(G)$ has local holomorphic sections everywhere in $\mathcal{T}(G)$. Fix $\phi \in \mathcal{T}(G)$. There exists a univalent function $f$ on $\mathbb{D}^{*}$ such that $\mathcal{S}_f = \phi$. Let $F$ be a Riemann mapping from $\mathbb{D}$ onto $\mathbb{C}\setminus\overline{f(\mathbb{D}^{*})}$ and $h = F^{-1}\circ f$ the conformal welding. Then we have $h \in SQS(\mathbb{S})$. Now the proof in \cite[P.199]{Le} gives that $h \in SQS(G)$. Let $f_{\mu} = F \circ E(h)$. Then $f_{\mu}$ is a quasiconformal extension of $f$ whose dilatation $\mu$ is in 
$\mathcal{M}(G)$. Thus  $S(\mu) = \phi$. Let $\Omega = f_{\mu}(\mathbb{D})$ and $\Omega^{*} = f_{\mu}(\mathbb{D}^{*})$. We have  
$\rho_{\Omega^{*}} (f_{\mu}(z))|f_{\mu}^{'}(z)| = (|z|^2 - 1)^{-1}$ for $z \in \mathbb{D}^{*}$.
Consider $U_{\epsilon}(\phi) = \{\psi \in \mathcal{B}(G): \|\psi - \phi\|_{\mathcal{B}} <  \epsilon\}$ for $\epsilon > 0$.
Then for each $\psi \in U_{\epsilon}(\phi)$ there exists a unique locally univalent function $f_{\psi}$ in $\mathbb{D}^{*}$ with $f_{\mu}(z) = z + \frac{b_{1}}{z} + \cdots$ as $z \rightarrow  \infty$
such that $\mathcal{S}_{f_{\psi}} = \psi$. 
Set $g_{\psi} = f_{\psi}\circ(f_{\mu})^{-1}$. Then $\mathcal{S}_{g_{\psi}} = ((\psi - \phi) \circ f_{\mu}^{-1})((f_{\mu}^{-1})^{'})^2$ and 
$\sup_{z \in \Omega^{*}} \rho_{\Omega^{*}}^{-2}(z) |\mathcal{S}_{g_{\psi}}(z)| = \|\psi - \phi\|_{B}$.  
Since $\psi, \phi \in \mathcal{B}(G)$, $\psi \circ g g^{' 2} = \psi$ and $\phi \circ g g^{' 2} = \phi$ for any $g \in G$. Hence, 
\begin{equation}\label{Schwarz}
\mathcal{S}_{g_{\psi}} \circ \widetilde{g} (\widetilde{g})^{'2} = \mathcal{S}_{g_{\psi}}, 
\end{equation}
for any 
$\widetilde{g} \in G_{\mu}$.
Since the inclusion map $i: \mathcal{B}(G) \rightarrow B(G)$ is continuous,  $\|\psi - \phi\|_{B} < C\epsilon$ for $\psi \in U_{\epsilon}(\phi)$.  
The Earle-Nag reflection \cite[P.263]{GL} associated with the curve $\Gamma = \partial \Omega$ is given by the formula 
$$
\gamma(z) = \begin{cases}
f \circ j \circ f_{\mu}^{-1}(z) = f \circ j \circ E(h)^{-1}\circ F^{-1}, &  z \in \Omega\\
z, & z \in \Gamma\\
\gamma^{-1}(z), & z \in \Omega^{*}
\end{cases}
$$
where $j(z) = 1/\bar{z}$, and \cite[P.265]{GL}  says 
\begin{equation}\label{gamma}
C_3^{-1}(\|\mu\|_{\infty}) \leqslant |\gamma(z) - z|^2 \rho_{\Omega^{*}}^{-2} (\gamma(z)) |\bar{\partial} \gamma(z)| \leqslant C_3(\|\mu\|_{\infty}).
\end{equation}
Under the condition that $\sup_{z \in \Omega^{*}} \rho_{\Omega^{*}}^{-2}(z) |\mathcal{S}_{g_{\psi}}(z)|$ is sufficiently small (when $\epsilon$ is sufficiently small),   Ahlfors \cite{Ah66}, Earle and Nag \cite{EN} (see also \cite[P.266]{GL}) proved that $g_{\psi}$ is univalent and can be extended to a quasiconformal mapping in the whole plane whose complex dilatation $\mu_{\psi}$ has the form
\begin{equation}\label{Omega}
\mu_{\psi}(z) = \frac{\mathcal{S}_{g_{\psi}}(\gamma (z))(\gamma (z) - z)^2 \bar{\partial} \gamma (z)}{2 + \mathcal{S}_{g_{\psi}}(\gamma (z))(\gamma (z) - z)^2 \partial{\gamma}(z)},    \, \, z \in \Omega.
 \end{equation}
Then by means of (\ref{gamma}) we have
\begin{equation} \label{mu}
|\mu_{\psi}(z)| \leqslant C_{4}(\|\mu\|_{\infty}) |\mathcal{S}_{g_{\psi}}(\gamma (z))| \rho_{\Omega^{*}}^{-2}(\gamma (z)) , \,\, z \in \Omega.
\end{equation}
Consequently, $f_{\psi} = g_{\psi} \circ f_{\mu}$ is univalent in $\mathbb{D}^{*}$ and has a quasiconformal extension to the whole plane whose complex dilatation $\nu_{\psi}$ is
\begin{equation}\label{nu}
\nu_{\psi} = \frac{\mu + (\mu_{\psi} \circ f_{\mu})\tau}{1 + \bar{\mu}(\mu_{\psi} \circ f_{\mu})\tau}, \,\,\, \tau = \frac{\overline{\partial f_{\mu}}}{\partial f_{\mu}}.
\end{equation}
According to Lemma \ref{composition}, (\ref{Schwarz}) implies that
$$\mu_{\psi} \circ \widetilde{g} \frac{\overline{(\widetilde{g})^{'}}}{(\widetilde{g})^{'}} = \mu_{\psi}.$$
Again we know that $\mu \circ g \frac{\bar{g^{'}}}{g^{'}} = \mu$. By direct computation, we can get 
\begin{equation}\label{g}
\nu_{\psi} \circ g \frac{\bar{g^{'}}}{g^{'}} = \nu_{\psi}.
\end{equation}
Now, it follows from (\ref{mu}) that
\begin{equation}
\begin{split}
|\mu_{\psi}(f_{\mu}(z))| & \leqslant C_4(\|\mu\|_{\infty}) |\mathcal{S}_{g_{\psi}} (\gamma (f_{\mu}(z)))| \rho_{\Omega^{*}}^{-2}( \gamma (f_{\mu}(z))) \\
& =  C_4(\|\mu\|_{\infty}) |\mathcal{S}_{g_{\psi}} (f_{\mu} (j(z)))| \rho_{\Omega^{*}}^{-2}( \gamma (f_{\mu}(z))) \\
& =  C_4(\|\mu\|_{\infty})  |\psi(j(z)) - \phi (j(z))| (1 - |j(z)|^2)^2 ,\\
\end{split}
\end{equation}
which implies that $\|\lambda_{\mu_{\psi}\circ f_{\mu}} \|_{c} \leqslant C_5 (\|\mu\|_{\infty}) \|\psi - \phi\|_{\mathcal{B}}$. 
Thus, $\mu_{\psi}\circ f_{\mu} \in \mathcal{M}(\mathbb{D})$,  and we conclude by (\ref{nu}) that $\nu_{\psi} \in \mathcal{M}(\mathbb{D})$.
Combining with (\ref{g}), we can get $\nu_{\psi} \in \mathcal{M}(G)$.
On the other hand, from (\ref{Omega}) and (\ref{nu}) it is easy to see that $\nu_{\psi}$ depends holomorphically  on $\psi$.
Since $S(\nu_{\psi}) = \psi$, we conclude that $\nu: U_{\epsilon}(\phi) \rightarrow \mathcal{B}(G)$ is a local holomorphic section to $S: \mathcal{M}(G) \rightarrow \mathcal{B}(G).$ This completes the proof.
\end{proof}

\begin{corollary}\label{Bers}
Via the Bers embedding $\beta$, $\mathcal{T}_S$ carries a natural complex structure. Under this complex structure, the Bers embedding 
$$\beta:  \mathcal{T}_S \rightarrow \mathcal{T}(G)$$
is biholomorphic.
\end{corollary}

\begin{proof}
First of all, $\beta: \mathcal{T}_S \rightarrow \mathcal{T}(G)$  is one-to-one by Theorem A.

Next, $S: \mathcal{M}(G) \rightarrow \mathcal{B}(G)$ is continuous by Theorem \ref{section}  and  
$\Phi: \mu \mapsto [\mu]$ which maps $\mathcal{M}(G)$ onto $\mathcal{T}_S$ is a projection. 
Thus, $\beta: \mathcal{T}_S \rightarrow \mathcal{B}(G)$ is continuous by the topological property of the projection mapping. 
We fix a point $\phi \in \mathcal{T}(G)$, consider $U_{\epsilon}(\phi) = \{\psi \in \mathcal{B}(G): \|\psi - \phi\|_{\mathcal{B}} <  \epsilon\}$ for $\epsilon > 0$.
According to the proof of Theorem \ref{section}, there exists a mapping $\nu: \psi \rightarrow \nu_{\psi}$ which maps $U_{\epsilon}(\phi)$ into $\mathcal{M}(G)$, and it is continuous. The projection  
$\Phi: \mathcal{M}(G) \rightarrow \mathcal{T}_S$  is also continuous. So the composition $\beta^{-1} = \Phi \circ \nu:  U_{\epsilon}(\phi) \rightarrow \mathcal{T}_S$ is continuous. Thus, $\beta: \mathcal{T}_S \rightarrow \mathcal{T}(G)$ is a homeomorphism,  i.e.,  $\mathcal{T}_S$ is actually homeomorphic to an open subset in the complex Banach space $\mathcal{B}(G)$. 

Hence, by using this result, we define on $\mathcal{T}_S$ a complex analytic Banach manifold structure.  The Bers embedding  
$\beta: \mathcal{T}_S \rightarrow \mathcal{B}(G)$ is biholomorphic under this complex structure.
\end{proof}

\begin{corollary}
The canonical projection
$$\Phi: \mathcal{M}(G) \rightarrow \mathcal{T}_S$$
is holomorphic, and it has local holomorphic sections everywhere in $\mathcal{T}_S$.
\end{corollary}
\begin{proof}
First of all, we have 
$$\Phi = \beta^{-1} \circ S.$$
Since $\beta^{-1}$ and $S$ are holomorphic, it follows that $\Phi$ is holomorphic.

Next, let us consider the mapping 
$\Phi \circ \nu \circ \beta $
which is the identity on $ \beta^{-1}(U_{\epsilon}(\phi)).$
According to the proof of Theorem \ref{section}, $\nu \circ \beta$ is the desired section.
\end{proof}

\begin{corollary}
$\Psi: \mathcal{T}_S \rightarrow SQS(G)$ is a homeomorphism. Consequently, $SQS(G)$ possesses a complex structure so that $\Psi: \mathcal{T}_S \rightarrow SQS(G)$ 
is a biholomorphic isomorphism.
\end{corollary}

\begin{proof}
We know that $\Psi: \mathcal{T}_S \rightarrow SQS(G)$ is one-to-one by Theorem A. Shen and Wei \cite{SW} have proved that $\Psi: \mathcal{T}_S \rightarrow SQS(G)$ 
is a homeomorphism when $G = \{I\}$. So, $\Psi: \mathcal{T}_S \rightarrow SQS(G)$ is a homeomorphism for any Fuchsian group $G$. The homeomorphism $\beta \circ \Psi^{-1}: SQS(G) \rightarrow \mathcal{T}(G)$ endows the spaces 
$SQS(G)$ with the structure of complex Banach manifolds modeled on the Banach space $\mathcal{B}(G)$. Naturally, under this complex structure, $\Psi: \mathcal{T}_S \rightarrow SQS(G)$ is a biholomorphic isomorphism.
\end{proof}

\section{G\^ateaux-differentiability of $e$ at the origin}
In this section, we will show that the Douady-Earle extension operator $e: SQS(\mathbb{S}) \to \mathcal{M}(\mathbb{D})$ is  G\^ateaux-differentiable at the origin, and  compute its differential.

Let $E(h)$ denote the Douady-Earle extension of a quasisymmetric homeomorphism $h$ on $\mathbb{S}$. 
The definition of $E(h)$ is very simple:  given $z \in \mathbb{D}$, $E(h)(z)$ is the unique $w \in \mathbb{D}$  such that 
\begin{equation*}
F_h(z, w) = \frac{1}{2\pi} \int_{0}^{2\pi} \frac{h(e^{iu}) - w}{1 - \bar{w}h(e^{iu})}\frac{1 - |z|^2}{|z - e^{iu}|^2} du = 0.
\end{equation*}
Coming to BMO-Teichm\"uller theory, Cui and Zinsmeister \cite{CZ} have shown that if $h \in SQS(\mathbb{S})$ then, if $\mu$ denotes the complex dilatation of the Douady-Earle extension $E(h)$, it holds that 
\begin{equation*}
\frac{|\mu|^2(z)}{1 - |z|} dxdy
\end{equation*}
is a Carleson measure in $\mathbb{D}$, which implies that the Douady-Earle extension operator $h \mapsto  \mu$ is a bijection from $SQS(\mathbb{S})$ onto its image in $\mathcal{M}(\mathbb{D})$. Lemma \ref{CM} combined with \cite[Proposition 4.2]{Se} shows continuity at $h = Id$  of the operator, but global continuity is still unknown. 

 Suppose $X$ and $Y$ are Banach spaces, $U \subset X$ is open, and $G: X \to Y$. The differential $dG(u; b)$ of $G$ at $u \in U$ in the direction $b \in X$ is defined as 
\begin{equation*}
dG(u; b) = \lim_{t \to 0}\frac{G(u + tb) - G(u)}{t} = \frac{d}{d t} G(u + tb)|_{t = 0}.
\end{equation*}
If the limit exists for all $b \in X$, and if $dG(u; \cdot): X \to Y$ is linear and continuous, then one says that $G$ is G\^ateaux differentiable at the point $u$.

In the following we show that the operator $e: \,h \mapsto  \mu$ from $SQS(\mathbb{S})$ onto its image in $\mathcal{M}(\mathbb{D})$ is G\^ateaux differentiable at the origin. In order to make this statement precise we first write,  for $h \in SQS(\mathbb{S})$, $h(e^{it}) = e^{i\phi(t)}$, where $\phi$ is an increasing homeomorphism of $[0, 2\pi]$ such that $\phi^{'}(t)$ is a Muckenhoupt weight. The topology of $SQS(\mathbb{S})$ is the one inherited from $BMO_{\mathbb{R}}(\mathbb{S})$ the space of real-valued  $2\pi$-periodic BMO functions:
\begin{equation*}
d(h_1, h_2) = \|\log \phi_1^{'} - \log \phi_2^{'}\|_{BMO}.
\end{equation*}
Put $b(t) = \log \phi^{'}(t)$. There is no loss of generality to assume that $\int_0^{2\pi} b(t) dt = 0$ (since any $\varphi \in BMO$ can be identified with $\varphi + \alpha$,  $\alpha $ constant).

Conversely, suppose now that $b \in BMO_{\mathbb{R}}(\mathbb{S})$ with $\int_0^{2\pi} b(t) dt = 0$. Let $t > 0$ be small. We define $c(t)$ as being the unique real number such that 
\begin{equation*}
\int_0^{2\pi} e^{t b(u) - t c(t)} du = 2\pi.
\end{equation*}
That is 
\begin{equation*}
c(t) = \log[(\frac{1}{2\pi}\int_0^{2\pi}e^{t b(u)} du)^{\frac{1}{t}}].
\end{equation*}
We have $\lim_{t \to 0} c(t) = \frac{1}{2\pi}\int_0^{2\pi}b(u) du = 0$. Hence we can write $h_t(e^{iu}) = e^{i \phi_t(u)}$, where 
\begin{equation*}
\begin{split}
\phi_t(u) &= \int_0^{u} e^{t b(v) - t c(t)} dv = u + \int_0^u [e^{t b(v) - t c(t)} - 1]dv\\
& = u + t\int_0^{u} b(v) dv + o(t).\\
\end{split}
\end{equation*}
Set $B(u) = \int_0^{u} b(v) dv$. Then
\begin{equation}\label{ht}
h_t(e^{iu}) = e^{i \phi_t(u)} = e^{iu} + t i e^{iu} B(u) + o(t).
\end{equation}

For $h_t \in SQS(\mathbb{S})$, set $F_{h_t}(z, w) = 0$. That is 
\begin{equation*}
F(t, z, w) = \frac{1}{2\pi} \int_{0}^{2\pi} \frac{h_t(e^{iu}) - w}{1 - \bar{w}h_t(e^{iu})}\frac{1 - |z|^2}{|z - e^{iu}|^2}du = 0.
\end{equation*}
Write $w = E(h_t)(z) = f_t (z)$ and denote by $\mu_t = (f_t)_{\bar{z}}/(f_t)_{z}$ the complex dilatation of $f_t$.

\begin{theorem}
The Douady-Earle extension operator $e: \,h \mapsto \mu$ from $SQS(\mathbb{S})$ onto its image in $\mathcal{M}(\mathbb{D})$ is G\^ateaux differentiable at the origin, and the differentiate at the origin of  $e$ in the direction $b \in BMO_{\mathbb{R}}(\mathbb{S})$ is 
\begin{equation*}
de(0; b) =  - \frac{(1 - |z|^2)^2}{2\pi i} \int_0^{2\pi} \frac{3 e^{2iu}B(u)}{(1 - \bar{z} e^{iu})^4} du.
\end{equation*}
\end{theorem}

\vspace{4mm}

\textbf{Remark:} We shall specify now what we mean by G\^ateaux differentiability at the origin of the Douady-Earle extension operator. 

According to the above argument, the origin $u = 0$ of the Banach space $BMO_{\mathbb{R}}(\mathbb{S})$ corresponds to the origin $h = Id$ of the space $SQS(\mathbb{S})$. Then the corresponding complex dilatation $\mu$ is equal to $0$.

Let $t > 0$ be small. For any $b \in BMO_{\mathbb{R}}(\mathbb{S})$, the strongly quasisymmetric homeomorphism $h_t$ can be defined as before. The complex dilatation $\mu_t$ denotes the image of $h_t$ under the Douady-Earle extension operator.

By saying that the operator $e: \,h \mapsto \mu$ from $SQS(\mathbb{S})$ onto its image in $\mathcal{M}(\mathbb{D})$ is G\^ateaux differentiable at the origin, we mean (by definition) that $\frac{d\mu_t}{dt}|_{t = 0}$ exists, and it is linear and continuous in $b$.

\vspace{4mm}

\begin{proof}
We compute $\mu_t$ using the implicit function theorem and the formula $F(t, z, f_t(z)) = 0$. We get (writing $F_z$ for $F_z (t, z, w)$, etc.) the system 
\begin{equation*}
F_{\bar{z}} + F_{\bar{w}} (\overline{f_t})_{\bar{z}} + F_w (f_t)_{\bar{z}} = 0, \qquad \bar{F}_{\bar{z}} + \bar{F}_{w} (f_t)_{\bar{z}} + \bar{F}_{\bar{w}} (\overline{f_t})_{\bar{z}} = 0,
\end{equation*}
whose solution is
\begin{equation*}
(f_t)_{\bar{z}} = \frac{\bar{F}_{\bar{z}}F_{\bar{w}} - F_{\bar{z}} \bar{F}_{\bar{w}}}{|F_w|^2 - |F_{\bar{w}}|^2}, \qquad 
(f_t)_z = \frac{\bar{F}_z F_{\bar{w}} - F_z \bar{F}_{\bar{w}}}{|F_w|^2 - |F_{\bar{w}}|^2},
\end{equation*}
and finally
\begin{equation*}
\mu_t = \frac{(f_t)_{\bar{z}}}{(f_t)_z} = \frac{\overline{(F_z)}F_{\bar{w}} - F_{\bar{z}}\overline{(F_w)}}{\overline{(F_{\bar{z}})} F_{\bar{w}} - F_z \overline{(F_w)}} \triangleq \frac{N}{D}.
\end{equation*}
Therefore,
\begin{equation*}
\frac{d \mu_t}{d t}|_{t = 0} = (\frac{1}{D} \dot{N} - \frac{N}{D^2} \dot{D})|_{t = 0}.
\end{equation*}

Recall that 
\begin{equation*}
F(t, z, w) = \frac{1}{2\pi} \int_{0}^{2\pi} \frac{h_t(e^{iu}) - w}{1 - \bar{w}h_t(e^{iu})}\frac{1 - |z|^2}{|z - e^{iu}|^2}du.
\end{equation*}
Then we have
\begin{equation}\label{Fzbar}
F_{\bar{z}}(t, z, w) = \frac{1}{2\pi} \int_{0}^{2\pi} \frac{h_t(e^{iu}) - w}{1 - \bar{w}h_t(e^{iu})}\frac{e^{-iu}}{(\bar{z} - e^{-iu})^2}du,
\end{equation}
and
\begin{equation}\label{Fz}
F_{z}(t, z, w) = \frac{1}{2\pi} \int_{0}^{2\pi} \frac{h_t(e^{iu}) - w}{1 - \bar{w}h_t(e^{iu})}\frac{e^{iu}}{(z - e^{iu})^2}du,
\end{equation}
and
\begin{equation}\label{Fwbar}
F_{\bar{w}}(t, z, w) = \frac{1}{2\pi} \int_{0}^{2\pi} \frac{h_t(e^{iu})(h_t(e^{iu}) - w)}{(1 - \bar{w}h_t(e^{iu}))^2}\frac{1 - |z|^2}{|z - e^{iu}|^2}du,
\end{equation}
and
\begin{equation}\label{Fw}
F_{w}(t, z, w) = - \frac{1}{2\pi} \int_{0}^{2\pi} \frac{1 - |z|^2}{(1 - \bar{w}h_t(e^{iu})) |z - e^{iu}|^2}du.
\end{equation}
Noting that $w = E(h)(z) = z$ if $h(z) = z$, we conclude that 
\begin{equation}\label{Fz0}
F_{\bar{z}}(t, z, w) |_{t = 0} = F_{\bar{z}}(0, z, z) = 0, \qquad F_{z}(t, z, w) |_{t = 0} = F_{z}(0, z, z) = \frac{1}{1 - |z|^2},
\end{equation}
and 
\begin{equation}\label{Fw0}
F_{\bar{w}}(t, z, w) |_{t = 0} = F_{\bar{w}}(0, z, z) = 0, \qquad F_{w}(t, z, w) |_{t = 0} = F_{w}(0, z, z) = - \frac{1}{1 - |z|^2}.
\end{equation}
Then
\begin{equation*}
N|_{t = 0} = [\overline{(F_z)}F_{\bar{w}} - F_{\bar{z}}\overline{(F_w)}]|_{t = 0} = 0.
\end{equation*}
Thus we have 
\begin{equation}\label{mut1}
\frac{d \mu_t}{d t}|_{t = 0} = \frac{1}{D} \dot{N}|_{t = 0} = (1 - |z|^2)^{-1} (\frac{d}{d t} F_{\bar{w}}(t , z, w) + \frac{d}{d t} F_{\bar{z}}(t , z, w))|_{t = 0}.
\end{equation}

We next compute $\frac{d}{d t} F_{\bar{w}}(t , z, w)|_{t = 0}$ and $\frac{d}{d t} F_{\bar{z}}(t , z, w)|_{t = 0}$. Let us start by computing $\frac{\partial f_t}{\partial t}(z)|_{t = 0}$.
Differentiating $F(t, z, f_t(z)) = 0$ and $\overline{F(t, z, f_t(z))} = 0$ with respect to $t$ using the implicit function theorem we get the system 
\begin{equation*}
\begin{split}
& F_t(0, z, z) + F_w(0, z, z) \frac{\partial f_t}{\partial t} (z)|_{t = 0} + F_{\bar{w}} (0, z, z) \overline{\frac{\partial f_t}{\partial t} (z)|_{t = 0}} = 0,\\
& \overline{F_t(0, z, z)} + \overline{F_{\bar{w}}(0, z, z)}  \frac{\partial f_t}{\partial t} (z)|_{t = 0} + \overline{F_w(0, z, z)} \overline{\frac{\partial f_t}{\partial t} (z)|_{t = 0}} = 0,\\
\end{split}
\end{equation*}
whose solution is 
\begin{equation*}
\frac{\partial f_t}{\partial t} (z)|_{t = 0} = \frac{F_{\bar{w}}(0, z, z) \overline{(F_t (0, z, z))} -   \overline{(F_w (0, z, z))} F_t (0, z, z)}{|F_w (0, z, z)|^2 - |F_{\bar{w}}(0, z, z)|^2}.
\end{equation*}
By (\ref{ht}) we have 
\begin{equation}\label{Ft0}
F_t (0, z, z) = \frac{1}{2\pi} \int_0^{2\pi} \frac{(1 - z \bar{z})^2 i e^{2iu} B(u)}{(e^{iu} - z)(1 - \bar{z} e^{iu})^3} du.
\end{equation}
It follows from (\ref{Fw0}) and (\ref{Ft0}) that 
\begin{equation}\label{ft0}
\frac{\partial f_t}{\partial t} (z)|_{t = 0} = \frac{(1 - z \bar{z})^3}{2\pi} \int_0^{2\pi} \frac{i e^{2iu} B(u)}{(e^{iu} - z)(1 - \bar{z} e^{iu})^3} du.
\end{equation}

The derivative at $t = 0$ of the function $F_{\bar{w}}(t, z, w)$ is 
\begin{equation*}
\frac{d}{d t} F_{\bar{w}} (t, z, w)|_{t = 0} = F_{\bar{w}t}(0, z, z) + F_{\bar{w}w}(0, z, z) \frac{\partial f_t}{\partial t} (z)|_{t = 0} + F_{\bar{w}\bar{w}}(0, z, z)\overline{\frac{\partial f_t}{\partial t} (z)|_{t = 0}}.
\end{equation*}
By means of (\ref{Fwbar}) we have 
\begin{equation*}
F_{\bar{w}t}(0, z, z) = \frac{1 - |z|^2}{2\pi} \int_0^{2\pi} \frac{i e^{3iu} B(u)}{(1 - \bar{z} e^{iu})^3} (\frac{1}{e^{iu} - z} + \frac{2\bar{z}}{1 - \bar{z} e^{iu}}  + e^{-iu}) du,
\end{equation*}
and
\begin{equation*}
F_{\bar{w}w}(0, z, z) = -\frac{z}{(1 - |z|^2)^2}, \qquad  F_{\bar{w}\bar{w}}(0, z, z) = 0.
\end{equation*}
Then
\begin{equation}\label{F1}
\begin{split}
\frac{d}{d t} F_{\bar{w}} (t, z, w)|_{t = 0} & = - \frac{1 - |z|^2}{2\pi i} \int_0^{2\pi} \frac{e^{3iu} B(u)}{(1 - \bar{z} e^{iu})^3} (\frac{1}{e^{iu} - z} + \frac{2\bar{z}}{1 - \bar{z} e^{iu}}  + e^{-iu}) du\\
& + z (1 - |z|^2) \frac{1}{2\pi i} \int_0^{2\pi} \frac{e^{2iu} B(u)}{(e^{iu} - z)(1 - \bar{z} e^{iu})^3} du.\\
\end{split}
\end{equation}

Similarly, 
\begin{equation*}
\frac{d}{d t} F_{\bar{z}} (t, z, w)|_{t = 0} = F_{\bar{z}t}(0, z, z) + F_{\bar{z}w}(0, z, z) \frac{\partial f_t}{\partial t} (z)|_{t = 0} + F_{\bar{z}\bar{w}}(0, z, z)\overline{\frac{\partial f_t}{\partial t} (z)|_{t = 0}}.
\end{equation*}
By (\ref{Fzbar}) we have 
\begin{equation*}
F_{\bar{z}t}(0, z, z) = \frac{1 - |z|^2}{2\pi} \int_0^{2\pi} \frac{i e^{2iu}B(u)}{(1 - \bar{z} e^{iu})^4} du,
\end{equation*}
and 
\begin{equation*}
F_{\bar{z}w}(0, z, z) = 0, \qquad F_{\bar{z}\bar{w}}(0, z, z) = 0.
\end{equation*}
Therefore, 
\begin{equation}\label{F2}
\frac{d}{d t} F_{\bar{z}} (t, z, w)|_{t = 0} = - \frac{1 - |z|^2}{2\pi i} \int_0^{2\pi} \frac{e^{2iu}B(u)}{(1 - \bar{z} e^{iu})^4} du.
\end{equation}

It follows from (\ref{mut1}) (\ref{F1}) and (\ref{F2}) that 
\begin{equation}\label{mut2}
\frac{d \mu_t}{d t}|_{t = 0} = - \frac{(1 - |z|^2)^2}{2\pi i} \int_0^{2\pi} \frac{3 e^{2iu}B(u)}{(1 - \bar{z} e^{iu})^4} du.
\end{equation}
It is easy to see $\frac{d \mu_t}{d t}|_{t = 0}$ is linear in $b$.

Finally we need to show continuity, i.e.  that 
\begin{equation}\label{FS}
\frac{|\frac{d \mu_t}{d t}|_{t = 0}|^2}{1 - |z|^2} dxdy \in CM(\mathbb{D})
\end{equation}
with norm controlled by $\|b\|_{BMO}^2$. 
Since each $b \in BMO_{\mathbb{R}}(\mathbb{S})$ with $\int_0^{2\pi} b(x) dx = 0$ is a real-valued periodic function with period $2\pi$, the Fourier series of the function $b(x)$ is given by 
\begin{equation*}
b(x) = \sum_{n > 0} b_n e^{inx} + \sum_{n > 0} b_{-n}e^{-inx}
\end{equation*}
where
\begin{equation*}
b_n = \frac{1}{2\pi}\int_0^{2\pi} b(x) e^{-inx} dx, \qquad b_{- n} = \frac{1}{2\pi}\int_0^{2\pi} b(x) e^{inx} dx.
\end{equation*}
Since $b$ is real-valued, we have $b_{-n} = \overline{b_n}$. Then
\begin{equation*}
b(x) = \sum_{n > 0} b_n e^{inx} + \overline{\sum_{n > 0} b_{n}e^{inx}}.
\end{equation*}
It follows from the relation $B(u) = \int_0^u b(x) dx$ that $B(u) \in BMO_{\mathbb{R}}(\mathbb{S})$ and 
\begin{equation*}
B(u) = -i \sum_{n > 0} \frac{1}{n} b_n e^{inu} + \overline{(-i \sum_{n > 0} \frac{1}{n} b_n e^{inu})}.
\end{equation*}
Set $F(e^{iu}) = -i \sum_{n > 0} \frac{1}{n} b_n e^{inu}$. Then $B(u) = F(e^{iu}) + \overline{F(e^{iu})}$.

To prove (\ref{FS}), write 
\begin{equation*}
\begin{split}
\int_0^{2\pi}\frac{e^{2iu}B(u)}{(1 - \bar{z} e^{iu})^4} du & = \int_0^{2\pi} \frac{e^{2iu}F(e^{iu})}{(1 - \bar{z} e^{iu})^4} du + \int_0^{2\pi}\frac{e^{2iu}\overline{F(e^{iu})}}{(1 - \bar{z} e^{iu})^4} du\\
& = I_1 + I_2.\\
\end{split}
\end{equation*}
The first term on the right hand side $I_1$ equals to $0$ by Cauchy Theorem. By Cauchy integral formula we have 
\begin{equation*}
I_2 = \overline{\int_0^{2\pi}\frac{e^{-2iu} F(e^{iu})}{(1 - z e^{-iu})^4} du} = \overline{-i \int_{\mathbb{S}} \frac{\zeta F(\zeta)}{(\zeta - z)^4} d\zeta} = \overline{\frac{\pi}{3} ((z F(z))^{'})^{''}}.
\end{equation*}
Noting that 
$$(z F(z))^{'} = F(z) + zF^{'}(z) = (-i \sum_{n > 0}\frac{1}{n}b_n z^n) + (-i \sum_{n > 0} b_n z^n) \in BMOA(\mathbb{D}).$$
By a result of Fefferman and Stein \cite{FS72}, since $z\mapsto zF'(z)+F(z)$ is in $BMOA(\mathbb{D})$,
\begin{equation*}
\frac{|\frac{d \mu_t}{d t}|_{t = 0}|^2}{1 - |z|^2} dxdy = \frac{1}{4} |((zF(z))^{'})^{''}|^2 (1 - |z|^2)^3 dxdy  \in CM(\mathbb{D})
\end{equation*}
with norm controlled by $\|b\|_{BMO}^2$.
\end{proof}

\end{document}